\documentclass{jonasart}

\usepackage[T5]{fontenc}
\usepackage{tikz}

\usetikzlibrary{arrows.meta, backgrounds}
\pgfdeclarelayer{edgelayer}
\pgfdeclarelayer{nodelayer}
\pgfsetlayers{background, edgelayer, nodelayer, main}
\tikzstyle{black}=[fill=white, draw=black, shape=circle]
\tikzstyle{forward}=[-{Stealth[scale=1.25]}]
\tikzstyle{twoArrows}=[{Stealth[scale=1.25]}-{Stealth[scale=1.25]}]

\DeclareMathOperator{\Aut}{Aut}
\DeclareMathOperator{\Conj}{Conj}
\DeclareMathOperator{\id}{id}
\DeclareMathOperator{\Inn}{Inn}
\DeclareMathOperator{\lab}{lab}
\DeclareMathOperator{\qnd}{qnd}
\DeclareMathOperator{\rack}{rack}
\DeclareMathOperator{\RMlt}{RMlt}

\newcommand{\inv}{^{-1}}
\newcommand{\Q}{\mathbb{Q}}
\newcommand{\Sch}{\mathrm{Sch}}
\newcommand{\tr}{\triangleleft}
\newcommand{\und}{\mathrm{und}}
\newcommand{\Z}{\mathbb{Z}}

\title{Graph quandles: Generalized Cayley graphs of racks and right quasigroups}

\authors{L\d\uhorn c Ta}

\authorinfo{University of Pittsburgh, United States}{ldt37@pitt.edu}

\abstract{This article lays the foundations for an analogue of geometric group theory that studies actions on graphs by right quasigroups, including racks and quandles. We study markings of graphs that realize racks, and we introduce (di)graph invariants based on such markings. We show that all right quasigroups are realizable by edgeless graphs and complete (di)graphs. Using Schreier (di)graphs, we also characterize Cayley (di)graphs of right quasigroups $Q$ that realize $Q$. In particular, all racks are realizable by their full Cayley (di)graphs. This solves two problems of Valeriy Bardakov. Finally, we give graph-theoretic characterizations of labeled Cayley digraphs of right-cancellative magmas, right-divisible magmas, right quasigroups, racks, quandles, involutory racks, and kei.}



\acknowledgments{I thank the anonymous referees for their careful reading and helpful comments on previous versions of this paper.}

\VOLUME{1}
\ISSUE{1}
\NUMBER{3}
\DOI{https://doi.org/10.46298/jonas.17215}
\editinfo{January 1, 2026}{March 24, 2026}{Mahender Singh}

\begin{document}

	\section {Introduction}

	\subsection{Motivating discussion}

	In 1992, Fenn and Rourke~\cite{fenn} introduced \emph{racks} to develop complete invariants of framed links. Racks are generalizations of \emph{quandles}, which Joyce~\cite{joyce} and Matveev~\cite{matveev} independently introduced in 1982 to develop complete invariants of unframed links. In turn, quandles generalize algebraic structures called \emph{kei}, which Takasaki~\cite{takasaki} introduced in 1942 to study Riemannian symmetric spaces. Racks and quandles play important roles in knot theory, low-dimensional topology, and quantum algebra, making it important to understand their structure. For general references on the theory, see~\cite{book,quandlebook}.

	Similarly to groups, there is a~rich literature on finite racks and quandles in algebra and topology (see~\cite{survey} for references), but studying infinite racks is more difficult. In recent years, this problem has created interest in developing a~more geometric understanding of racks and quandles~\cite{bardakov}. To that end, various authors have adapted methods from geometric group theory to study racks, including Cayley graphs~\cite{metrics}, certain metrics~\cite{metrics, kedra}, and even a~bounded cohomology theory for racks and quandles~\cite{kedra, kapari, saraf}.

	Racks, quandles, and groups are special classes of nonassociative algebraic structures called \emph{right quasigroups}, which are used to study column Latin squares~\cite{Latin}, smooth deformations of Lie group structures~\cite{diff}, and nonassociative generalizations of Hopf algebras~\cite{quantum}. To develop an analogue of geometric group theory for racks and quandles, it is therefore natural to study actions on geometric objects by right quasigroups rather than groups. In this paper, we take those geometric objects to be graphs, particularly labeled Cayley digraphs and other (di)graphs obtained from racks.

	\subsubsection{Racks as symmetries}
    
	In 2020, Bardakov~\cite{bardakov} introduced a~way to construct right quasigroups from \emph{markings} of graphs by graph automorphisms. Seeking geometric interpretations of racks, Bardakov posed the following questions about which marked graphs realize racks and how they relate to Cayley graphs of racks.
    
	\begin{problem}\label{prob1}
		Under what conditions does a~marked graph realize a~rack or a~quandle?
	\end{problem}

	\begin{problem}\label{prob2}
		Given a~rack or quandle $Q$, is there always a~marked graph $(\Gamma, R)$ that realizes $Q$? If so, can we choose $\Gamma$ to be a~Cayley graph of $Q$?
	\end{problem}
    
	We answer Bardakov's questions with the following results. In the following, we refer to directed graphs as \emph{digraphs} and simple undirected graphs as \emph{graphs}.
    
	\begin{theorem}[see Theorem~\ref{thm:1}]
		Let $\Gamma$ be a~\textup{(}di\textup{)}graph, denote its vertex set by $V$, and let $R:V\to \Aut\Gamma$ be a~marking \textup{(}resp.\ q-marking\textup{)} of $V$. Then the right quasigroup $V^\Gamma_R$ realized by $\Gamma$ is a~rack \textup{(}resp.\ quandle\textup{)} if and only if $R$ is a~magma homomorphism from $V^\Gamma_R$ to $\Conj(\Aut \Gamma)$.
	\end{theorem}

	\begin{proposition}[see Proposition~\ref{prop:edgeless}]
		All right quasigroups are realizable by edgeless graphs and complete \textup{(}di\textup{)}graphs.
	\end{proposition}

	\begin{theorem}[see Theorems~\ref{thm:marked1} and~\ref{thm:marked2}]
		Let $V_R$ be a~right quasigroup, and let $\Gamma$ be a~Cayley digraph \textup{(}resp.\ graph\textup{)} of $V_R$ with connection set $S\subseteq V$. Then $(\Gamma,R)$ realizes $V_R$ if and only if, for all $h,v\in V$ and $s\in S$, there exists an element $t\in S$ such that
		\[
			R_t R_h (v)=R_h R_s (v) \quad
            \text{(resp. } R_t^{\pm 1}R_v (w)=R_v R_s (w) \text{)}.
		\]
	\end{theorem}

	\begin{corollary}[see Corollary~\ref{cor:conj}]
		All racks are realizable by their full Cayley \textup{(}di\textup{)}graphs.
	\end{corollary}
    
	Theorem~\ref{thm:1} answers Problem~\ref{prob1}. Since all racks are right quasigroups, Proposition~\ref{prop:edgeless} answers a~generalized version of the first question in Problem~\ref{prob2}. Corollary~\ref{cor:conj} answers the second question in Problem~\ref{prob2}, while Theorems~\ref{thm:marked1} and~\ref{thm:marked2} answer generalized forms of the question. As an application of marked graphs, we also introduce two integer invariants $\mu_{\mathrm{rack}},\mu_{\mathrm{qnd}}$ of (di)graphs.

	\subsubsection{Characterization of labeled Cayley digraphs}
    
	Although this paper is the first to study labeled Cayley (di)graphs of right quasigroups in general, various authors have studied Cayley graphs of various classes of right quasigroups appearing in combinatorics, algebraic topology, and knot theory. For example, full Cayley digraphs of unital, fixed point-free right quasigroups have various applications in network theory~\cite{groupoid1}, and full Cayley graphs of right quasigroups have applications in categorical covering theory~\cite{categories}.

	Introduced by Winker~\cite{winker} in 1984, full Cayley graphs of racks help classify finite quotients of fundamental quandles of links~\cite{link3, link2} and generalizations of these quotients~\cite{link1}. Full Cayley graphs of racks can even be interpreted as 1-skeletons of CW complexes called \emph{extended rack spaces} and used to construct homotopy invariants of links~\cite{homology}. A very recent application of Cayley graphs of racks makes it possible to study infinite quandles via methods adapted from geometric group theory~\cite{metrics}.

	In this light, a~graph-theoretic rather than purely algebraic characterization of Cayley graphs of right quasigroups and racks is desirable. A question of Hamkins~\cite{hamkins} calls for such a~characterization for Cayley graphs of groups; in response, Caucal~\cite{caucal, caucal2} generalized this question to the settings of magmas and \emph{labeled digraphs} or \emph{labeled transition systems}, that is, digraphs with an assignment of directed edges (rather than vertices) to elements of a~distinguished \emph{labeling set}. Unlike Caucal, we assume the labeling set to be a~subset of the vertex set.
    
	\begin{problem}\label{prob3}
		Given a~full subcategory $\mathcal{C}$ of magmas (e.g., groups, monoids, quandles), are there graph-theoretic conditions that characterize labeled Cayley digraphs of objects in $\mathcal{C}$?
	\end{problem}
    
	Caucal answered Problem~\ref{prob3} for left quasigroups, quasigroups~\cite{caucal}, semigroups, and various classes of monoids, including groups~\cite{caucal2}. Note that the answers to Problem~\ref{prob3} for left and right quasigroups are distinct because the right-multiplication maps $R_v$ of left quasigroups are not necessarily permutations. Chishwashwa et al.\ addressed a~similar question for vertex-labeled Cayley digraphs of unital, fixed point-free right quasigroups~\cite{groupoid1}. In this paper, we answer Problem~\ref{prob3} for various classes of racks and right quasigroups.
    
	\begin{theorem}[see Theorems~\ref{thm:lab1} and~\ref{thm:lab2}]
		Let $\mathcal{Q}$ be the class of all labeled digraphs that are deterministic, source-complete, codeterministic, and target-complete. Then $\mathcal{Q}$ is precisely the class of labeled Cayley digraphs of right quasigroups.

		Moreover, there exist graph theoretic-conditions on $\mathcal{Q}$ that restrict it to the subclasses of labeled Cayley digraphs of racks and quandles.
	\end{theorem}
    
	\subsection{Structure of the paper}

	In Section~\ref{sec:prelims}, we discuss right quasigroups, racks, and quandles.
	In Section~\ref{sec:grph}, we discuss (di)graphs, marked graphs in the sense of Bardakov, Cayley (di)graphs of magmas in the sense of Caucal, and Schreier graphs of group actions.
	In Section~\ref{sec:ex}, we give examples of Cayley (di)graphs of right quasigroups and their markings.
	In Section~\ref{sec:3}, we answer Problem~\ref{prob1}. As an application, we introduce two rack-theoretic (di)graph invariants $\mu_{\mathrm{rack}},\mu_{\mathrm{qnd}}$ with general results for path graphs and cycle graphs.
	In Section~\ref{sec:4}, we answer Problem~\ref{prob2} and its analogues for right quasigroups and digraphs.
	In Section~\ref{sec:5}, we define labeled Cayley digraphs of magmas in the sense of Caucal, and we answer Problem~\ref{prob3} for right-cancellative magmas, right-divisible magmas, right quasigroups, racks, quandles, involutory racks, and kei.
	In Section~\ref{sec:6}, we propose directions for future research.

	\subsection*{Notation}
	Given a~positive integer $n\in\Z^+$, let $[n]$ denote the set $\{1,2,\dots,n\}$. Denote the symmetric group of $[n]$ by $S_n$ with its elements written in cycle notation, and denote the symmetric group of any other set $X$ by $S_X$. We also denote the composition of functions $\varphi:V\to W$ and $\psi:W\to X$ by $\psi\varphi$, and we denote the identity map on a~set $V$ by $\id_V$. For $n\geq 3$, let $D_n$ be the dihedral group of order $2n$. Given a~subset $S$ of a~group $G$, let $S\inv:=\{s\inv\mid s\in S\}$.

	\section{Algebraic preliminaries}\label{sec:prelims}

	We recall the definitions of right quasigroups, racks, and quandles.

	\subsection{Magmas} Right quasigroups are examples of more general algebraic structures called \emph{magmas}.

	\begin{definition}
		A \emph{magma} or \emph{groupoid} is a~pair $(V,R)$, denoted by $V_R$, where $V$ is a~set and $R$ is a~mapping from $V$ to the set of functions from $V$ to $V$. For all $v\in V$, we call the map $R_v:= R(v)$ a~\emph{right-multiplication map} or \emph{right-translation map}.

		Moreover, we say that $V_R$ is a~\emph{right quasigroup} if $R(V)\subseteq S_V$, that is, if all right-multiplication maps are permutations of $V$. We say that $R$ is a~\emph{right quasigroup structure} on $V$.
	\end{definition}

	\begin{remark}
		Although the notation $V_R$ is nonstandard, we use it because Bardakov~\cite{bardakov} formulated his open problems with it.
	\end{remark}

	\begin{example}\label{ex:right0}
		The (right) regular action $R:G\to S_G$ of a~group $G$, given by $R_h (g):=gh$, is a~right quasigroup structure on $G$. Thus, right quasigroups generalize groups.
	\end{example}

	\begin{example}\label{ex:right}
		If we take $G:=\Q$ to be the set of rational numbers in Example~\ref{ex:right0}, then the positive rational numbers $\Q^+$ are a~submagma of $\Q$ but not a~right quasigroup.
	\end{example}

	\begin{definition}
		Let $V_R$ and $W_T$ be magmas. A \emph{magma homomorphism} from $V_R$ to $W_T$ is a~function $\varphi:V\to W$ satisfying
		\[
            \varphi R_v=T_{\varphi(v)}\varphi
            \quad \text{for all } v \in V.
		\]
	\end{definition}

	\begin{remark}
		Some authors define magmas as sets $V$ equipped with some binary operation $\tr:V\times V\to V$; homomorphisms are then defined in the obvious way. For right quasigroups, racks, and quandles, the operation $\tr$ satisfies additional axioms like the right cancellation property.

		This convention is equivalent to our convention via the formula
		\[
			v\tr w=R_w(v).
		\]
		Indeed, it is easy to see that our notion of homomorphism is equivalent to the one using binary operations $\tr$. Our convention appears in~ \cite{survey, virtual}; we adopt this convention because it adapts more easily to graph-theoretic settings.
	\end{remark}
    
	Evidently, magmas and right quasigroups form categories. In particular, we can consider \emph{automorphism groups} $\Aut V_R$ of magmas.

	\subsection{Racks}
    
	Racks and quandles form important full subcategories of the category of right quasigroups.

	\begin{definition}
		Let $V_R$ be a~right quasigroup.
		\begin{itemize}
			\item The \emph{right-multiplication group} $\RMlt V_R$ is the subgroup of $S_V$ generated by all right-multiplication maps.
			\item We say that $V_R$ is a~\emph{rack} if every right-multiplication map $R_v$ is a~magma endomorphism. Concretely, this means that
			\begin{equation}\label{eq:rack}
				R_v R_w = R_{R_v (w)}R_v
			\end{equation}
			for all $v,w\in V$. In this case, we call $R$ a~\emph{rack structure} on $V$.
			\item Separately, we say that $V_R$ is \emph{involutory} if every right-multiplication map is an involution, that is, if $R_v^2 =\id_V$ for all $v\in V$.
		\end{itemize}
	\end{definition}

	\begin{remark}
		A right quasigroup $V_R$ is a~rack if and only if $\RMlt V_R$ is a~(normal) subgroup of $\Aut V_R$. In this case, some authors call $\RMlt V_R$ the \emph{inner automorphism group} of $V_R$ and denote it by $\Inn V_R$. Other authors denote $\RMlt V_R$ by $\operatorname{Mlt}_r V_R$.
	\end{remark}

	\begin{definition}
		Let $V_R$ be a~rack.
		\begin{itemize}
			\item We say that $V_R$ is a~\emph{quandle} if
			\(
			R_v (v)=v
			\)
			for all $v\in V$. In this case, we call $R$ a~\emph{quandle structure} on $V$.
			\item If $V_R$ is an involutory quandle, we call it a~\emph{kei}.
		\end{itemize}
	\end{definition}

	\subsubsection{Examples} We discuss some common examples of right quasigroups, racks, and quandles. See Section~\ref{sec:ex} for further examples.

	\begin{example}[\!\!{\cite[Example 2.13]{book}}]
		Let $G$ be a~union of conjugacy classes in a~group, and define $C:G\rightarrow S_G$ by sending any element $g\in G$ to the conjugation map $C_g$ defined by
		\[
			C_g(h):=ghg\inv.
		\]
		Then $\Conj G:=G_C$ is a~quandle called a~\emph{conjugation quandle} or \emph{conjugacy quandle}.

		If $G$ is a~group, then $\Conj G$ is a~kei if and only if $g^2\in Z(G)$ for all $g\in G$. In particular, not all quandles are involutory.
	\end{example}

	\begin{example}[\!\!{\cite[Example 99]{quandlebook}}]
		Let $V$ be a~set, fix a~permutation $\sigma\in S_V$, and define $R_v:=\sigma$ for all $v\in V$. Then the assignment $R$ is a~rack structure on $V$, and we call $V_\sigma:= V_R$ a~\emph{permutation rack} or \emph{constant action rack}.

		Note that $V_\sigma$ is a~quandle if and only if $\sigma=\id_V$, in which case we call $V_{\id_V}$ a~\emph{trivial quandle}. In particular, not all racks are quandles. Moreover, $V_\sigma$ is involutory if and only if $\sigma$ is an involution.
	\end{example}

	\begin{example}
		The (right) regular action of a~group $G$ is a~rack structure if and only if $G$ is the trivial group. In particular, not all right quasigroups are racks.
	\end{example}

	\subsubsection{Preliminary results}
	
    An alternative characterization of racks does the heavy lifting in solving Problem~\ref{prob1}; see Theorem~\ref{thm:1}.
    
	\begin{proposition}\label{prop:sam}
		Let $V_R$ be a~magma. Then $V_R$ is a~rack if and only if $R$ is a~magma homomorphism from $V_R$ to $\Conj S_V$.
	\end{proposition}

	\begin{proof}
		To prove necessity, suppose that $V_R$ is a~rack. Then $R(V)\subseteq S_V$,

		and Equation~\eqref{eq:rack} states that
		\[
			    RR_v(w)
                =R_{R_v(w)}
                =R_vR_wR_v\inv
                = C_{R_v}(R_w)
                = C_{R(v)}R(w)
		\]
		for all $v,w\in V$. Hence, $R$ is a~magma homomorphism from $V_R$ to $\Conj S_V$.

		To prove sufficiency, suppose that $R$ is a~magma homomorphism from $V_R$ to $\Conj S_V$.

		Similarly to before,
		\[
            R_{R_v(w)}
            =RR_v(w)
            =C_{R(v)}R(w)
            = C_{R_v}(R_w)
            = R_vR_wR_v\inv
		\]
		for all $v,w\in V$. Since $R_v\in S_V$ is bijective, we obtain Equation~\eqref{eq:rack}.
	\end{proof}

	In Sections~\ref{sec:ex} and~\ref{sec:4}, we employ a~necessary condition for a~right quasigroup to be a~rack.

	\begin{lemma}\label{lemma:conj}
		If $V_R$ is a~rack, then $R(V)$ is closed under conjugation.
	\end{lemma}

	\begin{proof}
		Equation~\eqref{eq:rack} states that $R_t R_h (v)=R_h R_s (v)$ for all $v,w\in V$.
	\end{proof}

	\section{Graph-theoretic preliminaries}\label{sec:grph}

	We discuss directed and undirected graphs, marked graphs as constructed by Bardakov~\cite{bardakov} and Cayley graphs of magmas as introduced by Caucal~\cite{caucal}. We also relate the latter to \emph{Schreier graphs} of group actions. Since we only discuss labeled digraphs in Section~\ref{sec:5}, we defer defining them until then.

	\subsection{Graphs and digraphs}
    
    We recall the graph-theoretic constructions of Bardakov from~\cite{bardakov}. Like Bardakov, we assume that all undirected graphs are simple, and we do not allow for digraphs to have multiple edges. However, we do allow for digraphs to have loops.

	\begin{definition}
		\phantom{-}
		\begin{itemize}
			\item A \emph{digraph} or \emph{directed graph} $\Gamma$ is a~pair $(V,E)$ where $V$ is a~set and $E$ is a~subset of $V\times V$. We say that $V$ is the \emph{vertex set} of $\Gamma$, and we say that $E$ is the (\emph{directed\textup{)} edge set} or \emph{arc set} of $\Gamma$, respectively.
			\item \emph{Simple undirected graphs}, which we only call \emph{graphs}, are defined similarly to digraphs, except that every element of $E$ is an unordered pair of vertices $\{v,w\}\subseteq V$ such that $v\neq w$.
			\item Given a~(di)graph $\Gamma$, we denote the vertex and edge sets of $\Gamma$ by $V(\Gamma)$ and $E(\Gamma)$, respectively. We say that $|V(\Gamma)|$ is the \emph{order} of $\Gamma$.
		\end{itemize}
	\end{definition}

	\begin{definition}
		Let $\Gamma=(V,E)$ be a~digraph. An \emph{automorphism} of $\Gamma$ is a~permutation $\varphi\in S_V$ such that $(\varphi(v),\varphi(w))\in E$ for all edges $(v,w)\in E$. Automorphisms of graphs are defined similarly.
	\end{definition}

	As with racks and quandles, digraphs and graphs form categories. In particular, we can consider automorphism groups $\Aut\Gamma$ of digraphs and graphs.

	\subsection{Marked graphs}
    
	We consider (di)graphs with \emph{markings} and \emph{q-markings} of their vertices as introduced by Bardakov~\cite{bardakov}.

	\begin{definition}
		Let $\Gamma$ be a~(di)graph with vertex set $V$.
		\begin{itemize}
			\item A \emph{marking} of $\Gamma$ is a~function $R:V\to \Aut \Gamma$ with each image notated as $R_v:=R(v)$. We say that the right quasigroup $V_R^\Gamma:=(V,R)$ is \emph{realized} by the \emph{marked \textup{(}di\textup{)}graph} $(\Gamma, R)$.
			\item Let $R$ be a~marking of $\Gamma$. We say that $R$ is also a~\emph{q-marking} of $\Gamma$ if $R_v (v)=v$ for all $v\in V$. In this case, we call $(\Gamma,R)$ a~\emph{q-marked graph}.
			\item Conversely, we say that a~right quasigroup $Q$ is \emph{realizable} by $\Gamma$ if there exists a~marking $R$ of $\Gamma$ such that $Q\cong V_R^\Gamma$.
		\end{itemize}
	\end{definition}

	\begin{remark}
		Bardakov calls $V_R^\Gamma$ a~\emph{graph groupoid}. We eschew this name to emphasize the fact that $V_R^\Gamma$ is not only a~magma but also a~right quasigroup.

		If $V_R^\Gamma$ is also a~rack (resp.\ quandle), Bardakov calls it a~\emph{graph rack} (resp.\ \emph{graph quandle}).
	\end{remark}

	\begin{remark}\label{obs:q}
		If $(\Gamma,R)$ is a~q-marked graph, then $V^\Gamma_R$ is a~rack if and only if $V^\Gamma_R$ is a~quandle.
	\end{remark}

	\begin{remark}
		Since markings of a~graph $\Gamma$ are simply functions $R:V(\Gamma)\to\Aut\Gamma$, the number of right quasigroup structures on $V(\Gamma)$ whose right-multiplication maps lie in $\Aut\Gamma$ equals
		\[
			\lvert\Aut\Gamma\rvert^{|V(\Gamma)|}.
		\]
	\end{remark}

	\begin{example}
		Given any (di)graph $\Gamma$, the trivial marking $v\mapsto \id_{V(\Gamma)}$ realizes a~trivial quandle.
	\end{example}

	\begin{example}
		Let $n\in\Z^+$ be a~positive integer, and let $\Gamma$ be the \emph{star graph} $K_{1,n}$ of order $n+1$; cf.\ Table~\ref{tab1} in Section~\ref{sec:3}. Then $\Aut \Gamma\cong S_n$ acts on $V:=V(\Gamma)$ by permuting the leaves.

		If $n=2$, it is easy to see that of the eight possible markings $R:V\to S_n$, exactly four are rack structures. Indeed, if $\ell_1,\ell_2\in V$ are the two leaves and $v\in V$ is the central vertex, then Equation~\eqref{eq:rack} forces $R_{\ell_1}=R_{\ell_2}$. For an example of a~marking of $K_{1,2}$ that does \emph{not} realize a~rack, see Example~\ref{ex:not}.
	\end{example}

	We give more substantial examples following our discussion of Cayley digraphs; see Section~\ref{sec:ex}.

	\subsubsection{Right quasigroups are realizable}
    
	We answer the first question in Problem~\ref{prob2}. Recall that a~digraph $\Gamma=(V,E)$ is called \emph{complete} if
	\[
		E =\{(v,w)\in V\times V\mid v\neq w\}.
	\]
	Complete graphs are defined similarly.
    
	\begin{proposition}\label{prop:edgeless}
		All right quasigroups are realizable by edgeless graphs and complete \textup{(}di\textup{)}graphs.
	\end{proposition}

	\begin{proof}
		Given a~right quasigroup $V_R$, let $\Gamma$ be an edgeless or complete (di)graph with vertex set $V$. Then $\Aut\Gamma=S_V$, so $R:V\to S_V$ is a~marking of $\Gamma$. Hence, $V_R^\Gamma=V_R$.
	\end{proof}

	\begin{remark}
		In Bardakov's original wording~\cite{bardakov}, Proposition~\ref{prop:edgeless} implies that every groupoid (resp.\ rack, quandle) is a~graph groupoid (resp.\ graph rack, graph quandle).
	\end{remark}
    
	\subsection{Cayley graphs}
    
    Having established Proposition~\ref{prop:edgeless}, it is natural to ask which racks are realized by (di)graphs with more intricate structures. To that end, we discuss Cayley (di)graphs of magmas as introduced by Caucal~\cite{caucal}.

	\begin{definition}\label{def:cay}
		Let $V_R$ be a~magma.
		\begin{itemize}
			\item The (\emph{generalized\textup{)} Cayley digraph} of $V_R$ with respect to a~subset $S\subseteq V$ is the digraph $\Gamma(V_R,S)$ with vertex set $V$ and edge set
			\[
				E:=\{(v,R_s(v))\mid v\in V,s\in S\}\subseteq V\times V.
			\]
			We say that $S$ is the \emph{connection set} of $\Gamma(V_R,S)$.
			\item The (\emph{generalized\textup{)} Cayley graph} of $V_R$ with respect to $S$, denoted by $\Gamma_{\und}(V_R,S)$, is the underlying (simple undirected) graph of $\Gamma(V_R,S)$.
			\item If $S=V$, then we call $\Gamma(V):=\Gamma(V_R,V)$ the \emph{full Cayley digraph} of $V_R$. The \emph{full Cayley graph} of $V_R$ is defined similarly.
		\end{itemize}
	\end{definition}

	\begin{remark}
		Unlike with Cayley graphs of \emph{groups} as typically considered in the literature, we do not assume that the connection set $S$ in Definition~\ref{def:cay} is symmetric. That is, we do not assume that $\id_V\notin R(S)$ or that $R(S)=R(S)\inv$; this is consistent with the definitions of Cayley graphs of quandles given in, for example,~\cite{winker, textbook, metrics}.

		We also do not assume that $R(S)$ is a~generating subset of $\RMlt V_R$; this is consistent with the definitions given in, for example,~\cite{caucal, caucal2}. This is why we call the Cayley graphs in Definition~\ref{def:cay} ``generalized.''
	\end{remark}
    
	\subsection{Schreier graphs}
	
    As Iwamoto et al. \cite{metrics} note, Cayley (di)graphs of right quasigroups are special types of Schreier (di)graphs, which are important objects of study in combinatorial and geometric group theory.

	\begin{definition}\label{def:sch}
		Let $T$ be a~subset of a~group $G$, and let $V$ be a~(left) $G$-set. The (\emph{generalized\textup{)} Schreier digraph} $\Gamma^{\Sch}(G,V,T)$ is the digraph with vertex set $V$ and edge set
		\[
			E:=\{(v,t\cdot v)\mid v\in V,\, t\in T\}\subseteq V\times V.
		\]
		The (\emph{undirected\textup{)} Schreier graph} $\Gamma_{\und}^{\Sch}(G,V,T)$ is defined similarly.
	\end{definition}

	\begin{remark}\label{lem:cayley}
		For all vertices $v,w\in V$ of $\Gamma:=\Gamma_{\und}^{\Sch}(G,V,T)$, the pair $\{v,w\}$ is an edge of $\Gamma$ if and only if there exists an element $t\in T$ such that $t\cdot w=v$ or $t\inv\cdot w=v$.
	\end{remark}

	\begin{remark}
		If $T$ is a~symmetric generating subset of $G$, then taking $V:=G$ with the (left) regular action in Definition~\ref{def:sch} recovers the traditional definition of the Cayley (di)graph of a~group.
	\end{remark}

	\begin{remark}\label{rmk:sch}
		Given a~right quasigroup $V_R$ and a~subset $S\subseteq V$, let $G:=\RMlt V_R$ and ${T:=R(S)}$. Then
		\[
			\Gamma^{\Sch}(G,V,T)=\Gamma(V_R,S),\quad\quad \Gamma_{\und}^{\Sch}(G,V,T)=\Gamma_{\und}(V_R,S).
		\]
		In the case that $V_R$ is a~quandle and $R(S)$ generates $G$, Iwamoto et al.\ called $\Gamma_{\und}(V_R,S)$ an \emph{inner graph}; in recent work, they applied the above equality to study quandles using methods from geometric group theory~\cite[Section~3]{metrics}.
	\end{remark}
	
    \subsubsection{Preliminary results}
    
	Our solutions to Problem~\ref{prob2} can be stated nicely in terms of Schreier (di)graphs; cf.\ Theorems~\ref{thm:marked1}--\ref{thm:marked2}.
    
	\begin{proposition}\label{prop:sch1}
		Let $\Gamma:=\Gamma^{\Sch}(G,V,T)$ be a~Schreier digraph with edge set $E$, and let $H$ be a~generating subset of $G$. The following are equivalent:
		\begin{enumerate}
			\item The action of $G$ on $V$ is also an action on $\Gamma$ by digraph automorphisms.
			\item For all elements $h\in H$, $v\in V$, and $s\in T$, there exists an element $t\in T$ such that
			\begin{equation}\label{eq:sch}
				th\cdot v=hs\cdot v.
			\end{equation}
		\end{enumerate}
	\end{proposition}

	\begin{proof}
		$(1)\implies(2)$: Let $h\in H$, $s\in T$, and $v\in V$, so $(v,s\cdot v)\in E$. By assumption, $(h\cdot v, hs\cdot v)\in E$, so there exists an element $t\in T$ that satisfies Equation~\eqref{eq:sch}.

		$(2)\implies (1)$:
		To show that $G$ acts on $\Gamma$ by digraph automorphisms, it suffices to show that $(h\cdot v,hs\cdot v)\in E$ for all elements $h\in H$ and directed edges $(v,s\cdot v)\in E$. By assumption, for all such elements and edges, there exists an element $t\in T$ that satisfies Equation~\eqref{eq:sch}. Hence, $(h\cdot v,hs\cdot v)\in E$.
	\end{proof}

	\begin{proposition}\label{prop:sch2}
		Let $\Gamma:=\Gamma_{\und}^{\Sch}(G,V,T)$ be a~Schreier graph, and let $H$ be a~generating subset of $G$. The following are equivalent:
		\begin{enumerate}
			\item The action of $G$ on $V$ is also an action on $\Gamma$ by graph automorphisms.
			\item For all elements $h\in H$, $v\in V$, and $s\in T$, there exists an element $t\in T$ such that one of the following equations holds:
			\begin{equation*}
				th\cdot v=hs\cdot v,\quad\quad h\cdot v=ths\cdot v.
			\end{equation*}
		\end{enumerate}
	\end{proposition}

	\begin{proof}
		The proof is nearly identical to that of Proposition~\ref{prop:sch1}; the only difference lies in using Remark~\ref{lem:cayley} in the obvious ways.
	\end{proof}

	\section{Motivating examples}\label{sec:ex} 
    
    We consider several examples of Cayley (di)graphs $\Gamma$ of right quasigroups $V_R$ and whether or not $(\Gamma,R)$ is a~marked graph; see also~\cite{bardakov} and~\cite[Section~1.15]{textbook}. These constructions serve as useful (counter)examples later in the paper.

	\begin{example}\label{ex:not}
		Equip the set $V=[3]$ with the right quasigroup structure given by
		\[
			R_1=\id_V,\quad R_2=(23),\quad R_3=(13).
		\]
		By Lemma~\ref{lemma:conj}, $V_R$ is not a~rack because
		\[
			R_3R_2R_3\inv=(12)\notin R(V).
		\]

		Figure~\ref{fig:not} depicts the full Cayley digraph $\Gamma(V_R )$ and the full Cayley graph $\Gamma_{\und}(V_R,V)$.
		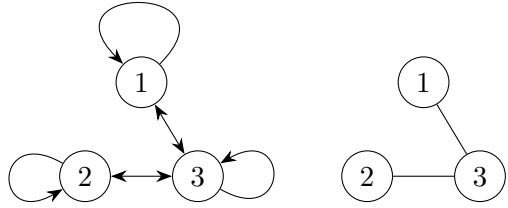
\begin{figure}[!ht]
			\centering
			\begin{tikzpicture}
				\begin{pgfonlayer}{nodelayer}
					\node [style=black] (0) at (3.75, 8.5) {1};
					\node [style=black] (1) at (3, 7.25) {2};
					\node [style=black] (2) at (4.5, 7.25) {3};
					\node [style=black] (3) at (7.5, 8.5) {1};
					\node [style=black] (4) at (6.75, 7.25) {2};
					\node [style=black] (5) at (8.25, 7.25) {3};
				\end{pgfonlayer}

				\begin{pgfonlayer}{edgelayer}
					\draw (4) to (5);
					\draw (5) to (3);
					\draw [style=twoArrows] (0) to (2);
					\draw [style=twoArrows] (1) to (2);
					\draw [style=forward, in=30, out=-30, loop] (2) to ();
					\draw [style=forward, in=-150, out=150, loop] (1) to ();
					\draw [style=forward, in=135, out=45, loop] (0) to ();
				\end{pgfonlayer}
			\end{tikzpicture}
			\caption{Full Cayley digraph and full Cayley graph of the right quasigroup from Example~\ref{ex:not}.}
			\label{fig:not}
		\end{figure}
		Evidently, $V_R$ is not realizable by either $\Gamma(V_R )$ or $\Gamma_{\und}(V_R,V)$; the only nontrivial (di)graph automorphism is $(12)$. In particular, $R$ is a~marking of neither $\Gamma(V_R )$ nor $\Gamma_{\und}(V_R,V)$.
	\end{example}

	\begin{example}\label{ex:3quandle}
		Equip the set $V=[3]$ with the quandle structure given by
		\[
			R_1=(23),\quad R_2=(13),\quad R_3=(12).
		\]
		Note that $V_R$ is a~kei.

		Let $S=\{1\}$. Figure~\ref{fig:3quandle} depicts the partial Cayley digraph $\Gamma(V_R,S)$, the full Cayley digraph $\Gamma(V_R )$, and the full Cayley graph $\Gamma_{\mathrm{und}}(V_R,V)$.
		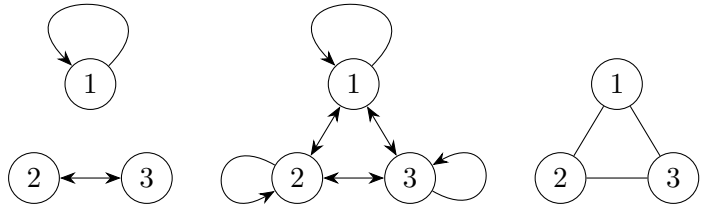
\begin{figure}[!ht]
			\centering
			\begin{tikzpicture}
				\begin{pgfonlayer}
					{nodelayer}
					\node [style=black] (0) at (3.75, 8.5) {1};
					\node [style=black] (1) at (3, 7.25) {2};
					\node [style=black] (2) at (4.5, 7.25) {3};
					\node [style=black] (3) at (0.25, 8.5) {1};
					\node [style=black] (4) at (-0.5, 7.25) {2};
					\node [style=black] (5) at (1, 7.25) {3};
					\node [style=black] (6) at (7.25, 8.5) {1};
					\node [style=black] (7) at (6.5, 7.25) {2};
					\node [style=black] (8) at (8, 7.25) {3};
				\end{pgfonlayer}

				\begin{pgfonlayer}
					{edgelayer}
					\draw [style=twoArrows] (0) to (1);
					\draw [style=twoArrows] (1) to (2);
					\draw [style=twoArrows] (0) to (2);
					\draw [style=forward, in=30, out=-30, loop] (2) to ();
					\draw [style=forward, in=-150, out=150, loop] (1) to ();
					\draw [style=forward, in=135, out=45, loop] (0) to ();
					\draw [style=twoArrows] (4) to (5);
					\draw [style=forward, in=135, out=45, loop] (3) to ();
					\draw (6) to (7);
					\draw (7) to (8);
					\draw (6) to (8);
				\end{pgfonlayer}
			\end{tikzpicture}
			\caption{Partial and full Cayley digraphs and full Cayley graph of the quandle from Example~\ref{ex:3quandle}.}
			\label{fig:3quandle}
		\end{figure}
		Although $R$ is not a~marking of $\Gamma(V_R,S)$ (or even $\Gamma_{\und}(V_R,S)$), it is a~marking of $\Gamma(V_R )$ and, hence, of $\Gamma_{\mathrm{und}}(V_R,V)$.
	\end{example}

	\begin{example}\label{ex:different}
		Nonisomorphic right quasigroups may share the same Cayley graphs and even the same Cayley digraphs. For example, let $V$ be the set $[3]$. Figure~\ref{fig:different} depicts the full Cayley digraph of the right quasigroup $V_R$ defined by
		\[
			R_1=(12),\quad R_2=(13),\quad R_3=(23),
		\]
		the full Cayley digraph of the permutation rack $V_{(123)}$, and the full Cayley graph shared by $V_R$ and $V_{(123)}$.
		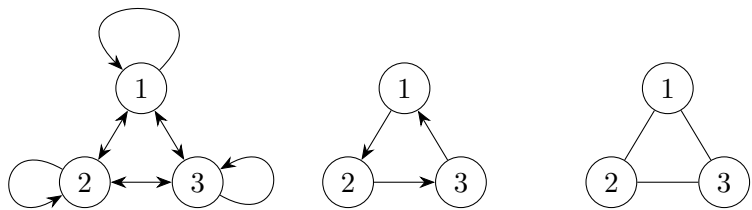
\begin{figure}[!ht]
			\centering
			\begin{tikzpicture}

				\begin{pgfonlayer}
					{nodelayer}
					\node [style=black] (0) at (3.75, 8.5) {1};
					\node [style=black] (1) at (3, 7.25) {2};
					\node [style=black] (2) at (4.5, 7.25) {3};
					\node [style=black] (3) at (0.25, 8.5) {1};
					\node [style=black] (4) at (-0.5, 7.25) {2};
					\node [style=black] (5) at (1, 7.25) {3};
					\node [style=black] (6) at (7.25, 8.5) {1};
					\node [style=black] (7) at (6.5, 7.25) {2};
					\node [style=black] (8) at (8, 7.25) {3};
				\end{pgfonlayer}

				\begin{pgfonlayer}
					{edgelayer}
					\draw (6) to (7);
					\draw (7) to (8);
					\draw (6) to (8);
					\draw [style=twoArrows] (4) to (5);
					\draw [style=forward, in=135, out=45, loop] (3) to ();
					\draw [style=forward, in=-150, out=150, loop] (4) to ();
					\draw [style=forward, in=30, out=-30, loop] (5) to ();
					\draw [style=twoArrows] (3) to (4);
					\draw [style=twoArrows] (3) to (5);
					\draw [style=forward] (0) to (1);
					\draw [style=forward] (1) to (2);
					\draw [style=forward] (2) to (0);
				\end{pgfonlayer}

			\end{tikzpicture}
			\caption{Full Cayley digraphs of the two right quasigroups from Example~\ref{ex:different} and their shared full Cayley graph.}
			\label{fig:different}
		\end{figure}
		Evidently, $V_R$ has the same full Cayley digraph as the quandle from Example~\ref{ex:3quandle}. Moreover, $V_R$ and $V_{(123)}$ have the same full Cayley graph. Of course, none of the right quasigroups in question are isomorphic; $V_R$ is not a~rack, $V_{(123)}$ is a~non-involutory rack, and the quandle from Example~\ref{ex:3quandle} is a~kei.
	\end{example}

	\begin{example}\label{ex:conj}
		Equip the set $V=[4]$ with the right quasigroup structure given by
		\[
			R_1=\id_V,\quad R_2=(1234),\quad R_3=(13)(24), \quad R_4=(24).
		\]
		Note that
		\[
			R_2R_4R_2\inv=(13)\notin R(V)\cup R(V)\inv.
		\]
		In particular, Lemma~\ref{lemma:conj} shows that $V_R$ is not a~rack.

		Figure~\ref{fig:conj} depicts the full Cayley digraph $\Gamma:=\Gamma(V_R )$ and the full Cayley graph $\Gamma_{\und}(V_R,V)$.
		\begin{figure}[!ht]
			\centering
			\begin{tikzpicture}
				\begin{pgfonlayer}
					{nodelayer}
					\node [style=black] (0) at (-1, 1) {1};
					\node [style=black] (1) at (1, 1) {2};
					\node [style=black] (2) at (1, -1) {3};
					\node [style=black] (3) at (-1, -1) {4};
					\node [style=black] (4) at (3, 1) {1};
					\node [style=black] (5) at (5, 1) {2};
					\node [style=black] (6) at (5, -1) {3};
					\node [style=black] (7) at (3, -1) {4};
				\end{pgfonlayer}

				\begin{pgfonlayer}
					{edgelayer}
					\draw [style=forward] (0) to (1);
					\draw [style=forward] (1) to (2);
					\draw [style=forward] (2) to (3);
					\draw [style=forward] (3) to (0);
					\draw (4) to (5);
					\draw (5) to (6);
					\draw (6) to (7);
					\draw (7) to (4);
					\draw (4) to (6);
					\draw (7) to (5);
					\draw [style=twoArrows] (0) to (2);
					\draw [style=twoArrows] (1) to (3);
					\draw [style=forward, in=135, out=45, loop] (0) to ();
					\draw [style=forward, in=135, out=45, loop] (1) to ();
					\draw [style=forward, in=-135, out=-45, loop] (2) to ();
					\draw [style=forward, in=-135, out=-45, loop] (3) to ();
				\end{pgfonlayer}
			\end{tikzpicture}
			\caption{Full Cayley digraph and full Cayley graph of the right quasigroup from Example~\ref{ex:conj}.}
			\label{fig:conj}
		\end{figure}
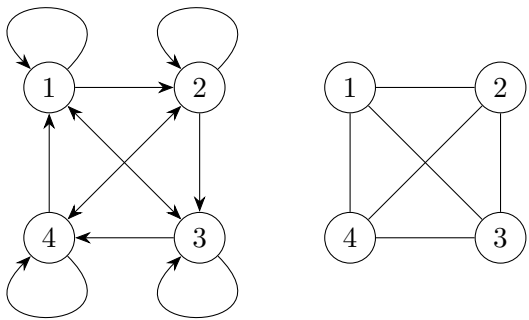
		Unlike in Example~\ref{ex:not}, $R$ is a~marking of $\Gamma$, so $(\Gamma,R)$ realizes $V^\Gamma_R$. (In fact, not only does $R$ land in $\Aut \Gamma$, but in fact $\RMlt V_R =\Aut \Gamma\cong D_4$.)
	\end{example}

	\begin{example}\label{ex:5quandle}
		Equip the set $V=[5]$ with the quandle structure given by
		\[
			R_1=(345),\quad R_2=(354),\quad R_3=(12)(45), \quad R_4=(12)(35), \quad R_5=(12)(34).
        \]
		Let $S=\{1\}$. Figure~\ref{fig:5quandle} depicts the partial Cayley digraph $\Gamma=\Gamma(V_R,S)$ and its underlying graph $\Gamma_{\und}(V_R,S)$.
		\begin{figure}[!ht]
			\centering
			\begin{tikzpicture}

				\begin{pgfonlayer}
					{nodelayer}
					\node [style=black] (0) at (0, 0) {1};
					\node [style=black] (1) at (1, -0.75) {2};
					\node [style=black] (2) at (0.5, -1.75) {3};
					\node [style=black] (3) at (-0.5, -1.75) {4};
					\node [style=black] (4) at (-1, -0.75) {5};
					\node [style=black] (5) at (4.5, 0) {1};
					\node [style=black] (6) at (5.5, -0.75) {2};
					\node [style=black] (7) at (5, -1.75) {3};
					\node [style=black] (8) at (4, -1.75) {4};
					\node [style=black] (9) at (3.5, -0.75) {5};
				\end{pgfonlayer}

				\begin{pgfonlayer}
					{edgelayer}
					\draw [style=forward, in=135, out=45, loop] (0) to ();
					\draw [style=forward, in=135, out=45, loop] (1) to ();
					\draw [style=forward] (2) to (3);
					\draw [style=forward] (3) to (4);
					\draw [style=forward] (4) to (2);
					\draw (7) to (8);
					\draw (8) to (9);
					\draw (9) to (7);
				\end{pgfonlayer}

			\end{tikzpicture}
			\caption{Partial Cayley digraph and underlying Cayley graph of the right quasigroup from Example~\ref{ex:5quandle}.}
			\label{fig:5quandle}
		\end{figure}
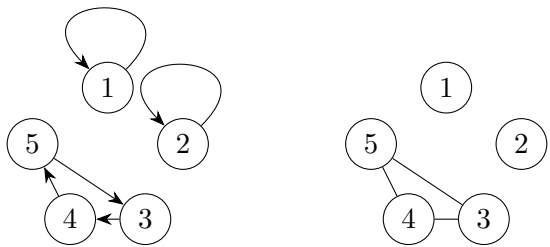
		Evidently, $R$ is not a~marking of $\Gamma(V_R,S)$, but it is a~marking of $\Gamma_{\und}(V_R,S)$.
		(Indeed, the automorphism group of the former is $\langle (345),(12)\rangle\cong\Z/6\Z$, which does not contain $R(V)$, while the automorphism group of the latter is $\RMlt V_R \cong S_3\times\Z/2\Z$.)
	\end{example}

	\section{From marked graphs to racks}\label{sec:3}
    
	In this section, we briefly deduce an answer to Problem~\ref{prob1}. As an application, we introduce two rack-theoretic invariants of graphs.

	\subsection{Solution to Problem~\ref{prob1}}
    
	Our alternative characterization of racks does the heavy lifting.
	\begin{theorem}\label{thm:1}
		Let $R$ be a~marking \textup{(}resp.\ q-marking\textup{)} of a~\textup{(}di\textup{)}graph $\Gamma$ with vertex set $V$. Then $V^\Gamma_R$ is a~rack \textup{(}resp.\ quandle\textup{)} if and only if $R$ is a~magma homomorphism from $V^\Gamma_R$ to $\Conj(\Aut\Gamma)$.
	\end{theorem}

	\begin{proof}
		The statement for markings follows immediately from Proposition~\ref{prop:sam}. Thus, the statement about q-markings follows from Remark~\ref{obs:q}.
	\end{proof}

	\begin{remark}\label{rmk:act}
		Theorem~\ref{thm:1} can be rephrased to state that $(\Gamma,R)$ realizes a~rack if and only if the group action of $\Aut\Gamma$ on $V(\Gamma)$ restricts to a~\emph{rack action} of $R(V(\Gamma))$ on $V(\Gamma)$; cf.~\cite{survey}.
	\end{remark}

	\begin{example}
		Let $\Gamma$ be the complete digraph $\overrightarrow{K_3}$ of order 3 (with loops). Then $\Gamma$ is the full Cayley digraph of the quandle from Example~\ref{ex:3quandle} and the non-rack right quasigroup from Example~\ref{ex:different}.

		These two right quasigroups are constructed using the same permutations of $V=[3]$, all of which happen to be automorphisms of $\Gamma$. However, the different choices of vertices that ${R:V^\Gamma_R\to\Conj(\Aut\Gamma)}$ assigns to those automorphisms determine whether or not $R$ is a~magma homomorphism and, hence, whether or not $V^\Gamma_R$ is a~rack.
	\end{example}

	\subsection{Application of Theorem~\ref{thm:1}}\label{sec:appl}
	
    Given a~(di)graph $\Gamma$, let $\mu_{\rack}(\Gamma)$ (resp.\ $\mu_{\qnd}(\Gamma)$) be the number of markings of $\Gamma$ that realize racks (resp.\ quandles). As an application of Theorem~\ref{thm:1}, we compute these numbers for several graphs and discuss how to compute them in general. This is motivated by Problem~\ref{prob1}.

	In light of Remark~\ref{rmk:act}, Theorem~\ref{thm:1} immediately implies the following.
    
	\begin{corollary}\label{cor:g-sets}
		Let $\Gamma_1$ and $\Gamma_2$ be (di)graphs whose automorphism groups are isomorphic to a~group $G$. If $V(\Gamma_1 )\cong V(\Gamma_2 )$ as $G$-sets, then $(\mu_{\rack}(\Gamma_1 ),\mu_{\qnd}(\Gamma_1 ))=(\mu_{\rack}(\Gamma_2 ),\mu_{\qnd}(\Gamma_2 ))$.
	\end{corollary}

	\begin{corollary}
		The numbers $\mu_{\rack}$ and $\mu_{\qnd}$ are (di)graph invariants.
	\end{corollary}
    
	To compute $\mu_{\rack}(\Gamma)$ given the action of $G:=\Aut \Gamma$ on $V:=V(\Gamma)$, Theorem~\ref{thm:1} implies that it suffices to count the number of functions $R:V\to G$ that are magma homomorphisms from $V_R$ to $\Conj G$. When $n:=|V|$ is finite, this calculation is possible via a~computer search. Namely, we use the \texttt{GRAPE} package~\cite{GRAPE} in \texttt{GAP}~\cite{GAP4} to compute the image $\rho(G)\cong G$ of the permutation representation $\rho:G\hookrightarrow S_n$ under the identification $V=[n]$. To compute $\mu_{\rack}(\Gamma)$ and $\mu_{\qnd}(\Gamma)$, we go through all $|G|^n$ possible functions $R:V\to \rho(G)$ and count how many of the corresponding right quasigroup structures satisfy the rack and quandle axioms.

	We provide an implementation of this exhaustive search algorithm in a~GitHub repository~\cite{code}. With this implementation, we were able to compute $\mu_{\rack}$ and $\mu_{\qnd}$ for complete graphs $K_n$ (equivalently, edgeless graphs), star graphs $K_{1,n-1}$, and cycle graphs $C_n$ for small values of $n$; see Table~\ref{tab1}. We also have the following general results for path graphs and cycle graphs.
	\begin{table}
		[]
		\centering
		\begin{tabular}
			{l|cccccccc}
			\multicolumn{1}{c|}{$n$} & 0 & 1 & 2 & 3 & 4 & 5 & 6 & 7 \\ \hline
			$K_n$ & $(1,1)$ & $(1,1)$ & $(2,1)$ & $(13,5)$ & $(114,36)$ &? &? &? \\
			$K_{1,n-1}$ & n/a & $(1,1)$ & $(2,1)$ & $(4,2)$ & $(31,13)$ & $(390,114)$ &? &? \\
			$C_n$ & n/a & n/a & n/a & $(13,5)$ & $(32,8)$ & $(41,7)$ & $(108,13)$ & $(113,9)$
		\end{tabular}
		\caption{Computations of $(\mu_{\rack},\mu_{\qnd})$ for complete graphs $K_n$, star graphs $K_{1,n-1}$, and cycle graphs $C_n$ for small values of $n$.}
		\label{tab1}
	\end{table}

	\begin{proposition}
		Let $P_n$ be a~path graph of order $n\geq 2$. Then
		\[
			\mu_{\rack}(P_n)=\begin{cases}
                2^k&\text{if }n=2k,\\
                2^{k+1}& \text{if }n=2k+1,\end{cases}\qquad \mu_{\qnd}(P_n)=\begin{cases}
                1&\text{if }n=2k,\\
               2& \text{if }n=2k+1.
           \end{cases}
		\]
	\end{proposition}

	\begin{proof}
		Let $V(P_n )=[n]$ be the vertices of $P_n$ in order, that is, the edges of $P_n$ are $(i,i+1)$ for all $1\leq i\leq n-1$. Then the nonidentity element $\sigma$ of $\Aut P_n\cong \Z/2\Z$ swaps $i$ and $n+1-i$ for all $1\leq i\leq \lfloor n/2\rfloor$. Therefore, a~function $R:V(P_n )\to\Conj(\Aut P_n )$ is a~magma homomorphism if and only if $R_i =R_{n+1-i}$ for all $1\leq i\leq n-1$. Since $\sigma$ has no fixed points if $n=2k$ and a~unique fixed point when $n=2k+1$, the claim follows from Theorem~\ref{thm:1}.
	\end{proof}

	\begin{proposition}\label{prop:app}
		Let $C_n$ be a~cycle graph of order $n\geq 3$, and let $\sigma(n)$ be the sum of all divisors of $n$. Then
		\[
			\mu_{\qnd}(C_n) = \sigma(n)+1.
		\]
	\end{proposition}
    
	The proof of Proposition~\ref{prop:app} uses the geometric interpretation of reflections in the dihedral group ${D_n\cong\Aut C_n}$. Although the proof is not terribly long, we defer it to Appendix~\ref{app:a} to avoid interrupting the flow of the paper.

	\begin{example}
		When $n=3$, five markings of the cycle graph $C_3$ realize quandles. In particular, let $V_R$ be the quandle from Example~\ref{ex:3quandle}. The full Cayley graph $\Gamma_{\und}(V_R,V)$ (which is isomorphic to $C_3$) depicted in Figure~\ref{fig:3quandle}, marked by the quandle structure $R:V\to S_3 =\Aut C_3$, realizes $V_R$. In the following section, we generalize this by showing that \emph{all} racks are realized by their full Cayley (di)graphs.
	\end{example}

	\section{From racks to marked graphs}\label{sec:4}

	In this section, we answer the second question in Problem~\ref{prob2}. We start by addressing a~generalized version of the question for right quasigroups and deduce solutions for racks afterward.

	\subsection{Results for right quasigroups}
    
	Propositions~\ref{prop:sch1} and~\ref{prop:sch2} do the heavy lifting in the directed and undirected cases, respectively.
    
	\begin{theorem}\label{thm:marked1}
		Let $S$ be a~subset of a~right quasigroup $V_R$, and let $\Gamma:=\Gamma(V_R,S)$. The following are equivalent:
		\begin{enumerate}
			\item $(\Gamma,R)$ is a~marked digraph that realizes $V_R$.
			\item $R$ is a~marking of $\Gamma$.
			\item For all $h,v\in V$ and $s\in S$, there exists an element $t\in S$ such that
			\[
				R_t R_h (v)=R_h R_s (v).
			\]
		\end{enumerate}
	\end{theorem}

	\begin{proof}
		$(1)\iff(2)$: Immediate.

		$(2)\iff(3)$: By definition, $H:=R(V)$ generates $G:=\RMlt V_R$. In light of Remark~\ref{rmk:sch}, the equivalence of (2) and (3) is a~special case of Proposition~\ref{prop:sch1}.
	\end{proof}

	\begin{theorem}\label{thm:marked2}
		Let $S$ be a~subset of a~right quasigroup $V_R$, and let $\Gamma:=\Gamma_{\und}(V_R,S)$. The following are equivalent:
		\begin{enumerate}
			\item $(\Gamma,R)$ is a~marked graph that realizes $V_R$.
			\item $R$ is a~marking of $\Gamma$.
			\item For all $h,v\in V$ and $s\in S$, there exists an element $t\in S$ such that one of the following equations holds:
			\[
				R_tR_h(v)=R_hR_s(v),\quad\quad R_h(v)=R_tR_hR_s(v).
			\]
		\end{enumerate}
	\end{theorem}

	\begin{proof}
		Similar to the proof of Theorem~\ref{thm:marked1}, with Proposition~\ref{prop:sch2} in place of Proposition~\ref{prop:sch1}.
	\end{proof}

	\begin{remark}
		The conditions of Theorem~\ref{thm:marked2} are strictly weaker than those of Theorem~\ref{thm:marked1}. Indeed, Example~\ref{ex:5quandle} gives an example of a~quandle $V_R$ and a~subset $S\subseteq V$ such that $R$ is a~marking of $\Gamma_{\und}(V_R,S)$ but not a~marking of $\Gamma(V_R,S)$.
	\end{remark}
    
	\subsection{Specialization to racks}
	
    By considering the third conditions in Theorems~\ref{thm:marked1}--\ref{thm:marked2}, we answer the second question in Problem~\ref{prob2} in its original form: All racks are realizable by their full Cayley (di)graphs.
	
    \begin{corollary}\label{cor:conj}
		In the setting of Theorem~\ref{thm:marked1} \textup{(}resp.\ Theorem~\ref{thm:marked2}\textup{)}, if conjugating $R(S)$ by elements of $R(V)$ lands in $R(S)$ \textup{(}resp.\ $R(S)\cup R(S)\inv$\textup{)}, then $(\Gamma,R)$ is a~marked digraph \textup{(}resp.\ graph\textup{)} realizing $V_R$. In particular, if $V_R$ is a~rack and $\Gamma$ is its full Cayley \textup{(}di\textup{)}graph, then $(\Gamma,R)$ realizes $V_R$.
	\end{corollary}

	\begin{proof}
		The conditions in the first claim directly imply the third conditions in Theorems~\ref{thm:marked1}--\ref{thm:marked2}. Therefore, the second claim follows from Lemma~\ref{lemma:conj}.
	\end{proof}

	\begin{remark}
		If $V_R$ is a~right quasigroup but not a~rack, then it is not true in general that conjugation in $R(V)$ lands in $R(V)$ (resp.\ $R(V)\cup R(V)\inv$). Nevertheless, the full Cayley digraph (resp.\ graph) may still satisfy the conditions of Theorem~\ref{thm:marked1} (resp.\ Theorem~\ref{thm:marked2}); see Example~\ref{ex:conj}.
	\end{remark}

	\begin{remark}
		Certainly, the full Cayley (di)graphs of non-rack right quasigroups do not satisfy the conditions of Theorems~\ref{thm:marked1}--\ref{thm:marked2} in general; see Example~\ref{ex:not}.
	\end{remark}

	\begin{remark}
		The partial Cayley (di)graphs of quandles do not satisfy the conditions of Theorems~\ref{thm:marked1}--\ref{thm:marked2} in general; see Examples~\ref{ex:3quandle} and~\ref{ex:5quandle}.
	\end{remark}
    
	\section{Characterization of labeled Cayley digraphs}\label{sec:5}
	
    In this section, we give a~graph-theoretic characterization of labeled Cayley digraphs of right-cancellative magmas, right-divisible magmas, right quasigroups, and certain classes of racks.

	\subsection{Preliminaries} First, we recall several definitions from nonassociative algebra and the theory of labeled digraphs.

	\subsubsection{Generalizations of right quasigroups}
	We recall two classes of magmas that generalize right quasigroups.

	\begin{definition}
		Let $V_R$ be a~magma. We say that $V_R$ is \emph{right-cancellative} (resp.\ \emph{right-divisible}) if $R_v$ is injective (resp.\ surjective) for all $v\in V$.
	\end{definition}

	\begin{remark}\label{obs:both}
		A magma is a~right quasigroup if and only if it is both right-cancellative and right-divisible.
	\end{remark}

	\begin{remark}
		If $V_R$ is a~finite magma, then $V_R$ is right-cancellative if and only if it is right-divisible.
	\end{remark}

	\begin{example}
		As in Example~\ref{ex:right}, the magma structure given by the addition maps $R_y (x):=x+y$ on the positive rational numbers $\Q^+$ yields a~right-cancellative magma that is not right-divisible.
	\end{example}

	\begin{example}
		Conversely, the magma structure given by $R_y (x):=x^3 -x$ on the set of rational numbers $\Q$ yields a~right-divisible magma that is not right-cancellative.
	\end{example}

	\subsubsection{Labeled digraphs}
    
	Following Caucal~\cite{caucal, caucal2}, we discuss labelings of digraph edges by vertices.

	\begin{definition}
		\phantom{-}
		\begin{itemize}
			\item A \emph{labeled digraph} is a~triple $\Gamma=(V,E,L)$ where $V$ is a~set, $L\subseteq V$ is a~subset, and $E\subseteq V\times L\times V$. Given $(v,\ell,w)\in E$, we say that $\ell$ is the \emph{label} of the \emph{edge} $(v,\ell,w)$. We say that $V$, $L$, and $E$ are the \emph{vertex, labeling,} and (\emph{labeled\textup{)} edge sets} of $\Gamma$, respectively.
			\item The \emph{labeled Cayley digraph} of a~magma $V_R$ with respect to a~\emph{connection set} $S\subseteq V$, denoted by $\Gamma_{\lab}(V_R,S)$, is the labeled digraph $(V,E,S)$ in which
			\[
				E=\{(v,s,R_s(v))\mid v\in V,\, s\in S\}.
			\]
		\end{itemize}
	\end{definition}

	\begin{remark}
		Labeled edges $(v,\ell,w)\in E$ can also be denoted by \emph{transitions} $v\xrightarrow{\ell}w$; for example, see Figures~\ref{fig:rack-cond}--\ref{fig:inv-cond}.
	\end{remark}

	\begin{definition}
		Let $\Gamma=(V,E,L)$ be a~labeled digraph.
		\begin{itemize}
			\item Let $\pi_1:E\to V\times L$ and $\pi_2:E\to L\times V$ be the projections from $E$ onto its first two and last two coordinates, respectively.
			\item We say that $\Gamma$ is \emph{deterministic} (resp.\ \emph{codeterministic}) if $\pi_1$ (resp.\ $\pi_2$) is injective.
			\item We say that $\Gamma$ is \emph{source-complete} or \emph{executable} (resp.\ \emph{target-complete} or \emph{coexecutable}) if $\pi_1$ (resp.\ $\pi_2$) is surjective.
		\end{itemize}
		Let $\mathcal{D}$ denote the class of deterministic, source-complete labeled digraphs.
	\end{definition}

	The following is immediate.
	\begin{remark}\label{obs:lab1}
		Let $V_R$ be a~magma, let $S\subseteq V$, and let $\Gamma=\Gamma_{\lab}(V_R,S)$. Then:
		\begin{itemize}
			\item $\Gamma$ lies in $ \mathcal{D}$.
			\item If $V_R$ is right-cancellative, then $\Gamma$ is codeterministic.
			\item If $V_R$ is right-divisible, then $\Gamma$ is target-complete.
		\end{itemize}
	\end{remark}
    
	Our first objective will be to prove a~converse to Remark~\ref{obs:lab1}; see Proposition~\ref{prop:lab1}.

	\subsection{Construction of $V^\Gamma_R$}

	Given an element $\Gamma=(V,E,L)$ of $\mathcal{D}$, we define a~magma structure $R$ on $V$ as follows. Since we do not consider marked graphs for the remainder of the paper, we denote this magma by $V^\Gamma_R$, overwriting the notation from previous sections.

	For each non-label vertex $v\in V\setminus L$, let $R_v:=\id_V$. Otherwise, for each label $\ell\in L$, define ${R_\ell:V\to V}$ as follows. Given a~vertex $v\in V$, the preimage $\pi_1\inv(v,\ell)$ contains a~unique element $(v,\ell,w)$ because $\Gamma$ is deterministic and source-complete. Thus, define $R_\ell(v):=w$. The assignment $v\mapsto R_v$ makes $V$ into a~magma $V^\Gamma_R$. Verifying the following is straightforward.
    
	\begin{remark}\label{obs:lab0}
		If $\Gamma=(V,E,L)$ is an element of $\mathcal{D}$, then $\Gamma$ is the labeled Cayley digraph $\Gamma_{\lab}(V^\Gamma_R,L)$.
	\end{remark}
    
	\subsection{First results}
	
    We apply our construction of $V^\Gamma_R$.
	
    \begin{proposition}\label{prop:lab1}
		Let $\Gamma=(V,E,L)$ be an element of $\mathcal{D}$.
		\begin{enumerate}
			\item If $\Gamma$ is codeterministic, then $V^\Gamma_R$ is right-cancellative.
			\item If $\Gamma$ is target-complete, then $V^\Gamma_R$ is right-divisible.
		\end{enumerate}
	\end{proposition}

	\begin{proof}
		(1): Suppose that $\Gamma$ is codeterministic, so $\pi_2$ is injective. We have to show for all $\ell\in V$ that $R_\ell$ is injective. If $\ell\notin L$, we are done. Otherwise, suppose for some $v,w\in V$ that
		\[
			R_\ell(v)=R_\ell(w)=:x.
		\]
		Since $\pi_2$ is injective, $\pi_2\inv(\ell,x)$ contains at most one element. Since $\pi_2\inv(\ell,x)$ contains $(v,\ell,x)$ and $(w,\ell,x)$, we obtain $v=w$, as desired.

		(2): Suppose that $\Gamma$ is target-complete, so $\pi_2$ is surjective. We have to show for all $\ell\in V$ that $R_\ell$ is surjective. If $\ell\notin L$, we are done. Otherwise, let $w\in V$. By hypothesis, $\pi_2\inv(\ell,w)$ contains an edge of $\Gamma$, say $(v,\ell,w)$. Hence, $R_\ell(v)=w$.
	\end{proof}

	Henceforth, let $\mathcal{Q}$ denote the subclass of $\mathcal{D}$ whose elements are also codeterministic and target-complete. The following answers Problem~\ref{prob3} for right-cancellative magmas, right-divisible magmas, and right quasigroups.
	\begin{theorem}\label{thm:lab1}
		A labeled digraph is the labeled Cayley digraph of a~right-cancellative magma \textup{(}resp.\ right-divisible magma, right quasigroup\textup{)} if and only if it is an element of $\mathcal{D}$ that is codeterministic \textup{(}resp.\ target-complete, contained in $\mathcal{Q}$\textup{)}.
	\end{theorem}

	\begin{proof}
		The first two claims follow from Remarks~\ref{obs:lab1}--\ref{obs:lab0} and Proposition~\ref{prop:lab1}. Therefore, the third claim follows from Remark~\ref{obs:both}.
	\end{proof}

	\subsection{Specialization to racks}

	Next, we specialize Theorem~\ref{thm:lab1} to racks, quandles, involutory right quasigroups, involutory racks, and kei. To that end, we introduce several graph-theoretic conditions corresponding to the labeled Cayley digraphs of objects in these categories.

	Recall from Remark~\ref{obs:lab0} and Proposition~\ref{prop:lab1} that each element $\Gamma=(V,E,L)$ in $\mathcal{Q}$ is the labeled Cayley digraph of the right quasigroup $V^\Gamma_R$ with respect to the connection set $L$.

	\begin{definition}\label{def:rack-cond}
		Let $\Gamma=(V,E,L)$ be an element of $ \mathcal{Q}$.
		\begin{itemize}
			\item We say that $\Gamma$ satisfies the \emph{first rack condition} if
			\[
				R_{\ell_1}(w_2)=R_{R_{\ell_1}(\ell_2)}(w_1)
			\]
            for all $v\in V$ such that $(v,\ell_1,w_1 ),(v,\ell_2,w_2 )\in E$. See Figure~\ref{fig:rack-cond} for a~visualization. Note that we do not assume that $w_1\neq w_2$.
			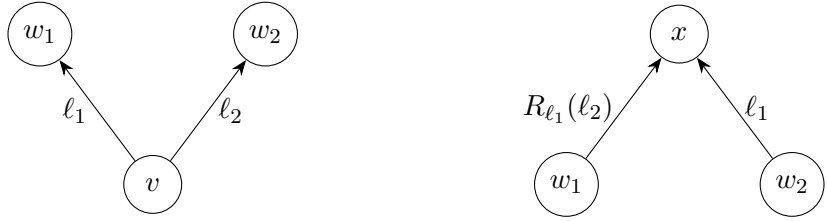
\begin{figure}
				\centering
				\begin{tikzpicture}
					\begin{pgfonlayer}
						{nodelayer}
						\node [style=black, minimum size=2em] (1) at (-1, -1) {$v$};
						\node [style=black, minimum size=2em] (2) at (0.5, 1) {$w_2$};
						\node [style=black, minimum size=2em] (3) at (-2.5, 1) {$w_1$};
						\node [style=black, minimum size=2em] (4) at (6, 1) {$x$};
						\node [style=black, minimum size=2em] (6) at (7.5, -1) {$w_2$};
						\node [style=black, minimum size=2em] (7) at (4.5, -1) {$w_1$};
					\end{pgfonlayer}

					\begin{pgfonlayer}
						{edgelayer}
						\draw [style=forward] (1) to node[left]{$\ell_1$} (3);
						\draw [style=forward] (1) to node[right]{$\ell_2$}(2);
						\draw [style=forward] (7) to node[left]{$R_{\ell_1}(\ell_2 )$}(4);
						\draw [style=forward] (6) to node[right]{$\ell_1$}(4);
					\end{pgfonlayer}

				\end{tikzpicture}
				\caption{The \emph{first rack condition} from Definition~\ref{def:rack-cond} states that for all subgraphs of the form on the left, there also exists a~subgraph of the form on the right, where $x:=R_{\ell_1}(w_2 )$.}
				\label{fig:rack-cond}
			\end{figure}
            
			\item We say that $\Gamma$ satisfies the \emph{second rack condition} if for all edges $(v,\ell,w)$ and non-label vertices $x\in V\setminus L$ such that $R_\ell(x)\in L$ is a~label, then $w$ has a~loop labeled by $R_{\ell}(x)$; i.e.,
			\[
				(w, R_{\ell}(x), w)\in E.
			\]
			See Figure~\ref{fig:rack-cond-2} for a~visualization.
			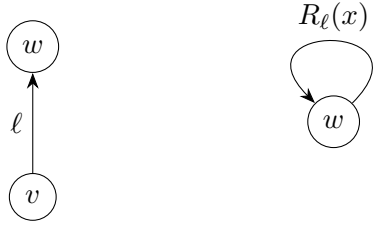
\begin{figure}
				\centering
				\begin{tikzpicture}

					\begin{pgfonlayer}
						{nodelayer}
						\node [style=black] (0) at (-2, -1) {$v$};
						\node [style=black] (1) at (2, 0) {$w$};
						\node [style=black] (2) at (-2, 1) {$w$};
					\end{pgfonlayer}

					\begin{pgfonlayer}
						{edgelayer}
						\draw [style=forward, in=135, out=45, loop] (1) to node[above]{$R_{\ell}(x)$}();
						\draw [style=forward] (0) to node[left]{$\ell$}(2);
					\end{pgfonlayer}

				\end{tikzpicture}
				\caption{The \emph{second rack condition} from Definition~\ref{def:rack-cond} states that for all subgraphs of the form on the left and non-label vertices $x\in V\setminus L$, if $R_\ell(x)\in L$ is a~label, then there exists a~subgraph of the form on the right.}
				\label{fig:rack-cond-2}
			\end{figure}
            
			\item We say that $\Gamma$ is \emph{label-idempotent} if $R_\ell(\ell)=\ell$ for all $\ell\in L$; that is, each vertex $\ell$ contained in the labeling set $L$ has a~loop $(\ell,\ell,\ell)\in E$ labeled by $\ell$.
			\item We say that $\Gamma$ satisfies the \emph{label-involutory} if $R^2_\ell(v)=v$ for all $v\in V$ and $\ell\in L$; that is, the edge set $E$ can be partitioned into loops and cycles of length 2 having the form
			\[
				\{(v,\ell,w),(w,\ell,v)\}.
			\]
			See Figure~\ref{fig:inv-cond} for a~visualization.
			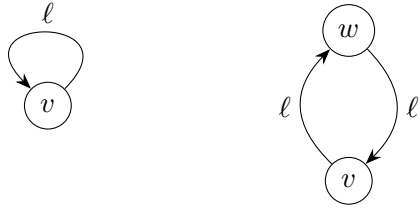
\begin{figure}
				\centering
				\begin{tikzpicture}

					\begin{pgfonlayer}
						{nodelayer}
						\node [style=black] (0) at (2, -1) {$v$};
						\node [style=black] (1) at (-2, 0) {$v$};
						\node [style=black] (2) at (2, 1) {$w$};
					\end{pgfonlayer}

					\begin{pgfonlayer}
						{edgelayer}
						\draw [style=forward, in=135, out=45, loop] (1) to node[above]{$\ell$}();
						\draw [style=forward, bend left=45, looseness=1.25] (0) to node[left]{$\ell$}(2);
						\draw [style=forward, bend left=45, looseness=1.25] (2) to node[right]{$\ell$}(0);
					\end{pgfonlayer}

				\end{tikzpicture}
				\caption{The \emph{label-involutory} condition from Definition~\ref{def:rack-cond} states that subgraphs of the forms on the left and the right partition the edge set.}
				\label{fig:inv-cond}
			\end{figure}
		\end{itemize}
	\end{definition}

	\subsubsection{Result} 
    
    We answer Problem~\ref{prob3} for racks, quandles, and involutory right quasigroups. This also yields answers for involutory racks and kei in the obvious way.
    
	\begin{theorem}\label{thm:lab2}
		Let $\Gamma=(V,E,L)$ be a~labeled digraph.
		\begin{enumerate}
			\item $\Gamma$ is the labeled Cayley digraph of a~rack if and only if $\Gamma$ lies in $ \mathcal{Q}$ and satisfies the two rack conditions.
			\item $\Gamma$ is the labeled Cayley digraph of a~quandle if and only if $\Gamma$ lies in $ \mathcal{Q}$, satisfies the two rack conditions, and is label-idempotent.
			\item $\Gamma$ is the labeled Cayley digraph of an involutory right quasigroup if and only if $\Gamma$ is a~label-involutory element of $ \mathcal{Q}$.
		\end{enumerate}
	\end{theorem}

	\begin{proof}
		By Theorem~\ref{thm:lab1}, we can assume that $\Gamma\in\mathcal{Q}$. As noted before, this inclusion implies that $\Gamma=\Gamma_{\lab}(V^\Gamma_R,L)$.

		(1): First, suppose that $V^\Gamma_R$ is a~rack; we show that $\Gamma$ satisfies the rack conditions. For all vertices $v\in V$ and edges ${(v,\ell_i,w_i )\in E}$, we have $w_i =R_{\ell_i}(v)$. It follows from Equation~\eqref{eq:rack} that
		\[
			R_{\ell_1}(w_2)=R_{\ell_1}R_{\ell_2}(v)=R_{R_{\ell_1}(\ell_2)}R_{\ell_1}(v)=R_{R_{\ell_1}(\ell_2)}(w_1).
		\]
		Hence, $\Gamma$ satisfies the first rack condition. Next, let $(v,\ell,w)\in E$, and let $x\in V\setminus L$ satisfy $R_\ell(x)\in L$. Since $R_x =\id_V$ and $R_\ell(v)=w$, Equation~\eqref{eq:rack} yields
		\[
			R_{R_\ell(x)}(w)
            = R_{R_\ell(x)}R_\ell(v)
            = R_\ell R_x(v)
            = R_\ell(v)
            = w.
		\]
		Since $\Gamma=\Gamma_{\lab}(V^\Gamma_R,L)$, it follows that $(w,R_{\ell}(x),w)\in E$, so $\Gamma$ satisfies the second rack condition.

		Conversely, suppose that $\Gamma\in \mathcal{Q}$ satisfies the two rack conditions; we show that $V^\Gamma_R$ is a~rack. We have to verify that
		\begin{equation*}
			R_{\ell_1}R_{\ell_2}(v)=R_{R_{\ell_1}(\ell_2 )}R_{\ell_1}(v)
		\end{equation*}
		for all $\ell_1,\ell_2,v\in V$. If $\ell_1\notin L$, we are done. Next, suppose that $\ell_1\in L$ and $\ell_2\notin L$. If $R_{\ell_1}(\ell_2 )\notin L$, we are done. Otherwise, applying the second rack condition to the edge $(v,\ell_1,R_{\ell_1}(v))$ yields the desired equality. Finally, if $\ell_1,\ell_2\in L$, then $(v,\ell_i,w_i )\in E$ with $w_i:=R_{\ell_i}(v)$. Since $\Gamma$ satisfies the first rack property,
		\[
			R_{\ell_1}R_{\ell_2}(v)=R_{\ell_1}(w_2)=R_{R_{\ell_1}(\ell_2)}(w_1)=R_{R_{\ell_1}(\ell_2)}R_{\ell_1}(v).
		\]
		(2): By the previous claim, it suffices to show that $V^\Gamma_R$ is a~quandle if and only if $\Gamma$ satifies the label-idempotence condition. But this is clear from the construction of $V^\Gamma_R$.

		(3): Clear from the construction of $V^\Gamma_R$.
	\end{proof}

	\begin{corollary}\label{cor:kei}
		Let $\mathcal{I}$ be the subclass of $\mathcal{Q}$ whose elements are label-involutory and satisfy the two rack conditions. Then $\mathcal{I}$ (resp.\ the subclass of $\mathcal{I}$ whose elements are label-idempotent) is precisely the class of labeled Cayley digraphs of involutory racks (resp.\ kei).
	\end{corollary}
    
	\section{Open questions}\label{sec:6}
	
    We conclude by proposing directions for future work.
	First, Problem~\ref{prob1} motivates the following.
	
    \begin{problem}\label{prob:compute}
		Compute $\mu_{\rack}$ and $\mu_{\qnd}$ for more families of (di)graphs.
	\end{problem}

	\begin{problem}\label{prob:compute2}
		Add more entries to Table~\ref{tab1}.
	\end{problem}
    
	A more computationally efficient implementation of the algorithm described in Subsection~\ref{sec:appl} will help in addressing Problems~\ref{prob:compute}--\ref{prob:compute2}.

	The existence of nonisomorphic graphs that satisfy the hypotheses of Corollary~\ref{cor:g-sets}---for example, any non-self-complementary graph $\Gamma$ and its complement $\overline{\Gamma}$---shows that the pair $(\mu_{\rack},\mu_{\qnd})$ is not a~complete invariant of graphs. In this light, it is interesting to ask the following.
    
	\begin{problem}
		Under what conditions do nonisomorphic graphs share the same values of $\mu_{\rack}$ and/or $\mu_{\qnd}$ without satisfying the hypotheses of Corollary~\ref{cor:g-sets}?
	\end{problem}
    
	Finally, Problem~\ref{prob3} and recent work applying geometric group theory to quandle theory~\cite{metrics, kedra, saraf, kapari} motivate further analogues of Theorem~\ref{thm:lab2}.
	
    \begin{problem}
		Characterize labeled Cayley digraphs of various classes of racks (e.g., medial racks, Latin quandles, fundamental racks of framed links).
	\end{problem}

	\begin{problem}
		Characterize unlabeled Cayley graphs of right quasigroups, racks, and quandles.
	\end{problem}

	\begin{problem}
		Define and characterize Cayley (di)graphs of classes of racks equipped with extra structure (e.g., generalized Legendrian racks~\cite{ta}, multi-virtual quandles~\cite{virtual}, symmetric racks~\cite{ta2}).
	\end{problem}
    
	\appendix

	\section{Proof of Proposition~\ref{prop:app}}\label{app:a}

	\subsection{Preliminaries}
	Let $n\geq 3$. Recall that a~subgroup of the dihedral group $D_n$ is called a~\emph{reflection subgroup} if it is either the trivial subgroup or generated by reflections.

	Recall that the automorphism group of the cycle graph $C_n$ is isomorphic to $D_n$. Therefore, by a~1975 result of Cavior~\cite{cavior}, proving the following will also prove Proposition~\ref{prop:app}.
    
	\begin{proposition}\label{prop:app1}
		Let $\mathcal{S}$ be the set of reflection subgroups of $D_n$, and let $\mathcal{M}$ be the set of markings $R$ of $C_n$ such that $V^{C_n}_R$ is a~quandle. Then there exists a~bijection $\varphi:\mathcal{S}\to\mathcal{M}$.
	\end{proposition}
	
    \subsection{Construction of $\varphi$}
	
    We construct a~function $\varphi:\mathcal{S}\to\mathcal{M}$ geometrically. Given a~reflection subgroup $G\in \mathcal{S}$, let $\varphi(G)$ be the marking $R:V\to D_n$ defined as follows. For each vertex $v\in V$, if $v$ lies on the axis of a~reflection $\psi\in G$ (which is necessarily unique if it exists), then let $R_v:=\psi$. Otherwise, let $R_v:=\id_V$.

	\begin{lemma}
		If $G\in \mathcal{S}$, then $\varphi(G)\in \mathcal{M}$.
	\end{lemma}

	\begin{proof}
		By construction, $R_v (v)=v$ for all $v\in V$. It remains to show that $R$ is a~rack structure on $V$. That is, we have to show that
		\[
			R_vR_wR_v\inv= R_{R_v(w)}
		\]
		for all vertices $v,w\in V$. But this is geometrically clear: Since $R_w$ is either the identity map or the reflection about the axis $\ell$ containing $w$, the composition $R_v R_w R_v\inv$ is either the identity map or the reflection about the axis $R_v (\ell)$, i.e., the axis containing $R_v (w)$. This transformation is precisely $R_{R_v (w)}$.
	\end{proof}

	\subsection{Bijectivity of $\varphi$}
    
	We construct an inverse map $\varphi\inv:\mathcal{M}\to\mathcal{S}$ as follows. Given a~quandle structure $R\in\mathcal{M}$, each right-multiplication map $R_v$ is either the identity map or a~reflection. This is because every rotation in $D_n$ has no fixed points. Therefore, defining
	\[
		\varphi\inv(R):=\RMlt V_R=\langle R_v\mid v\in V\rangle
	\]
	yields a~function $\varphi\inv:\mathcal{M}\to\mathcal{S}$. Verifying that $\varphi$ and $\varphi\inv$ are mutually inverse is straightforward. This completes the proof of Proposition~\ref{prop:app1} and, hence, that of Proposition~\ref{prop:app}.

    %%% REFERENCES %%%
    
\end{document}